\documentclass[reqno]{amsart}

\usepackage{amsmath}
\usepackage{amsfonts}
\usepackage{amssymb,enumerate}
\usepackage{amsthm}
\usepackage[all]{xy}
\usepackage{rotating}
\usepackage{hyperref}

\theoremstyle{plain}
\newtheorem{lem}{Lemma}[section]
\newtheorem{cor}[lem]{Corollary}
\newtheorem{prop}[lem]{Proposition}
\newtheorem{thm}[lem]{Theorem}

\theoremstyle{definition}
\newtheorem{defn}[lem]{Definition}
\newtheorem{ex}[lem]{Example}

\newtheorem{problem}[lem]{Problem}
\newtheorem{disc}[lem]{Remark}
\newtheorem{rmk}[lem]{Remark}

\newcommand{\HH}{\operatorname{H}}
\newcommand{\Hom}{\operatorname{Hom}}

\newcommand{\im}{\operatorname{Im}}

\newcommand{\Ker}{\operatorname{Ker}}

\newcommand{\ideal}[1]{\mathfrak{#1}}
\newcommand{\m}{\ideal{m}}
\newcommand{\n}{\ideal{n}}
\newcommand{\p}{\ideal{p}}
\newcommand{\q}{\ideal{q}}
\newcommand{\fm}{\ideal{m}}
\newcommand{\fn}{\ideal{n}}
\newcommand{\fp}{\ideal{p}}
\newcommand{\fq}{\ideal{q}}

\newcommand{\fa}{\ideal{a}}

\newcommand{\ol}{\overline}
\newcommand{\wti}{\widetilde}

\newcommand{\bbz}{\mathbb{Z}}

\newcommand{\bbq}{\mathbb{Q}}
\newcommand{\bbr}{\mathbb{R}}

\newcommand{\xra}{\xrightarrow}
\newcommand{\xla}{\xleftarrow}
\newcommand{\onto}{\twoheadrightarrow}
\newcommand{\into}{\hookrightarrow}

\renewcommand{\geq}{\geqslant}
\renewcommand{\leq}{\leqslant}

\newcommand{\Ext}[4][R]{\operatorname{Ext}_{#1}^{#2}(#3,#4)}

\newcommand{\Otimes}[3][R]{#2\otimes_{#1}#3}
\renewcommand{\Hom}[3][R]{\operatorname{Hom}_{#1}(#2,#3)}	
\newcommand{\Tor}[4][R]{\operatorname{Tor}^{#1}_{#2}(#3,#4)}

\newcommand{\ssm}{\smallsetminus}

\newcommand{\Supp}{\operatorname{Supp}}

\newcommand{\la}{\lambda}

\numberwithin{equation}{lem}

\begin{document}

\bibliographystyle{amsplain}

\author{Jason Boynton}

\email{jason.boynton@ndsu.edu}

\urladdr{http://math.ndsu.nodak.edu/faculty/boynton.shtml}

\author{Sean Sather-Wagstaff}

\email{sean.sather-wagstaff@ndsu.edu}

\urladdr{http://www.ndsu.edu/pubweb/\~{}ssatherw/}

\address{Department of Mathematics,
North Dakota State University Dept \# 2750,
PO Box 6050,
Fargo, ND 58108-6050
USA}

\thanks{Sean Sather-Wagstaff was supported in part by a grant from the NSA}

\title{Coherence conditions in flat regular pullbacks}

\date{\today}



\begin{abstract}
We investigate the behavior of four coherent-like
conditions in regular conductor squares.  In particular, we find necessary
and sufficient conditions in order that a pullback ring be a finite
conductor ring, a coherent ring, a generalized GCD ring, or quasi-coherent
ring. As an application of these results, we are able to determine exactly
when the ring of integer-valued polynomials determined by a finite subset
possesses one of the four coherent-like properties.
\end{abstract}

\maketitle

\section{Introduction}\label{sec150509a}

Throughout this paper, the term ``ring'' is short for ``commutative ring with identity'', and ``module'' is short for ``unital module''.

\

Over the last half century, there has been an abundance of research
dedicated to the study of the transference of various ring and
ideal-theoretic properties in pullback constructions.  It is well-known
that these pullback constructions provide a rich source of (counter)examples
in commutative algebra; see~\cite{luc} for a survey.  In his book \cite{G3}, Gilmer popularized a very special case of a pullback called the $D+M$
construction.  Gilmer's construction begins with a valuation domain $V$
containing a retract field $K$, meaning that $V=K+M$ for some maximal ideal $M$ of $V$. Let $D$ be a subring of $K$, and form the subring $D+M\subset V$.
In~\cite{dobbs:wDMc}, Dobbs and Papick find necessary and sufficient
conditions on $K$ and $D$ (or $M$) in order that $D+M$ is a coherent ring. 
In \cite{brewerandrutter}, Brewer and Rutter dropped the valuation condition on $T$ and
found similar conditions on the constituent rings so that the ring $D+M$ is
coherent.  In \cite{gabelli:ccp}, Houston and Gabelli offer improved results on the
transference of coherence and other coherent-like conditions by removing the
assumption that the domain $T$ contains a retract field.  

The purpose of this article is to investigate the transference of
coherent-like conditions in a more general setting that is very similar to
the pullback construction of \cite{FHP}; see Problem~\ref{prob150614a}.  We are primarily concerned with
finding conditions under which the (quasi) coherent, finite conductor, and
generalized GCD properties ascend and descend in a pullback.  The
construction central to our study, is described here explicitly. Start with
a ring surjection $\eta _{1}\colon T\twoheadrightarrow B$ and an inclusion
of rings $\iota _{1}\colon A\hookrightarrow B$ with $B\neq 0$, hence $A\neq 0
$. Let $R$ denote the pullback of these maps, that is, the subring of $A\times T$ consisting of all elements $(a,t)$ such that $\iota _{1}(a)=\eta
_{1}(t)$. The natural maps $\eta _{2}\colon R\twoheadrightarrow A$ and $\iota _{2}\colon R\hookrightarrow T$ yield a commutative diagram of ring
homomorphisms 
\begin{gather}
\begin{split}
\xymatrix{
R\ \ar@{^(->}[r]^-{\iota_2} \ar@{->>}[d]_-{\eta_2}
& T\ar@{->>}[d]^-{\eta_1} \\ 
A \ \ar@{^(->}[r]^-{\iota_1} &B}
\end{split}
\tag{$\square$}
\end{gather}
such that $\operatorname{Ker}(\eta _{2})=\operatorname{Ker}(\eta _{1})$. The common ideal $\operatorname{Ker}(\eta _{i})$ is the largest common ideal of $R$ and $T$; it is
denoted $C$ and called the \textit{conductor} of $T$ into $R$. When $C$
contains a $T$-regular element, we say that the conductor square $(\square )$
is \textit{regular}.

Conductor squares can also be built as follows. Let $T$ be a commutative
ring with subring $R$, and suppose that $R$ and $T$ have a common, non-zero
ideal. We call the largest common ideal $C$ the \textit{conductor} of $T$
into $R$. Setting $A=R/C$ and $B=T/C$, we obtain a commutative diagram $(\square )$ which is a conductor square.

It is common in the study of pullback constructions to assume that $T$ is an
integral domain and that $C$ is a maximal ideal of $T$. However, important
examples are obtained by allowing zero-divisors in the pullback square. For
example, let $D$ be an integral domain with field of fractions $K$, and let $E=\{e_{1},\ldots ,e_{r}\}\subset D$. Setting $T=K[X]$ and $C=(X-e_{1})\cdots
(X-e_{r})K[X]$, we have $B=T/C\cong \prod_{i=1}^{r}K$. Using $A=\prod_{i=1}^{r}D$ in the conductor square, we get 
$$R=\operatorname{Int}(E,D)=\{g\in K[X]\mid g(E)\subset D\}$$ the ring of integer-valued
polynomials on $D$ determined by the subset $E$. Observe that the rings $A$
and $B$ are not integral domains.  In fact, Chapman and Glaz have proposed
the following open question.

\begin{problem}[\protect{\cite[Problem 50]{CG}}]
\label{prob150614a}
Study the
ring and ideal-theoretic properties that transfer in a conductor square
where the conductor ideal is not maximal (or even prime) in the extension
ring.
\end{problem}

This paper is organized as follows. 
In Section~\ref{sec150509aa}, we
provide the relevant definitions and some background results for pullbacks
in which the extension $\iota _{2}\colon R\hookrightarrow T$ is flat.  In Section~\ref{sec150419a}, 
we show that, in a non-trivial pullback square, the conductor $C$ is \emph{never} finitely generated over $R$, and $\iota _{1}\colon A\hookrightarrow B$ is 
\emph{never} a faithfully flat extension of
rings, whenever $\iota _{2}\colon R\hookrightarrow T$ is flat.  In light of~\cite{boynton:rpb,gabelli:ccp}, this suggests that if $R$ is a
coherent-like ring defined by a regular conductor square of the type $\left(\square \right)$, then either $T$ is $R$-flat or $C$ (and $T$) is finitely
generated over $R$. 

The remaining sections, we assume that $T$ is $R$-flat in the regular
conductor square~$\left( \square \right)$.  In Section~\ref{sec150612a}, we assume that
the conductor $C$ is principal in $T$ and that a unitary type of condition
similar to~\cite{mcquillan} holds in order to find necessary and
sufficient conditions on the constituent rings $A$ and $T$ so that $R$ is a
finite conductor ring.  The proofs of these results all extend naturally to
the coherent ring case.  We conclude the section by showing that if $E$ is
finite, then the ring $\operatorname{Int}(E,D)$ is a finite conductor ring if and
only if $D$ is a finite conductor ring. We note that a similar result
holds for $\operatorname{Int}(E,D)$ in the coherent case.  In Sections~\ref{sec150612b} and~\ref{sec150612c},
we prove similar results for generalized GCD rings and quasi-coherent rings
in a flat regular conductor square. 

In contrast to the ideal-theoretic methods utilized in~\cite{gabelli:ccp}, our proofs rely slightly more on module theory. We make
frequent use of the results found in~\cite{glaz:ccr,glaz:fcp}.

\section{Background}\label{sec150509aa}

This section focuses on foundational notions and technical lemmas for the sequel.

\begin{rmk}\label{lem150509a}
Consider the regular conductor square $\left( \square \right)$. 
Then $T$ is a domain if and only if $R$ is a domain.
Indeed, since $R$ is a subring of $T$, one implication is easy (and does not use the regularity of $(\square)$.
The other implication follows from the fact that $T$ is an overring of $R$, by~\cite[Proposition 2.5(i)]{boynton:rpb}.
\end{rmk}

\begin{rmk}\label{rmk150615a}
It is straightforward to show that the following conditions on the  conductor square $\left( \square \right)$ are equivalent:
\begin{enumerate}[(i)]
\item $A=0$;
\item $C=R$; 
\item $T=R$;
\item $A=B$;
\item $T=C$; and
\item $B=0$.
\end{enumerate}
A \emph{trivial conductor square} is one that satisfies these equivalent conditions. 
We are only interested in non-trivial conductor squares.
\end{rmk}

\begin{defn}
A flat ring homomorphism $R\to T$ is a \emph{flat epimorphism} if the natural multiplication map $\mu\colon\Otimes TT\to T$ is 
bijective (i.e., injective).
\end{defn}

The next lemma 
shows that many of our results conform to the ``general format''.

\begin{lem}\label{lem150510a}
Consider the regular conductor square $\left( \square \right)$. 
If $T$ is flat over $R$, then $R\to T$ is a flat epimorphism.
\end{lem}

\begin{proof}
It suffices to show that the natural multiplication map $\mu\colon\Otimes TT\to T$ is injective.
Let $c\in C$ be $T$-regular.
The maps $T\xra ccT\xra\subseteq C\xra\subseteq R$ are injective.
Since $T$ is flat over $R$, the explains the monomorphisms in the top row of the following commutative diagram.
$$\xymatrix{
\Otimes TT\ \ar@{^(->}[r]^{\Otimes cT}\ar[d]_\mu
&\Otimes {(cT)}T\ \ar@{^(->}[r]
&\Otimes RT \ar[d]^{\cong} \\
T\ar[rr]^-c&&T.}$$
The unspecified vertical isomorphism is the natural one (tensor-cancellation).
It follows that $\mu$ is injective, as desired.
\end{proof}

\begin{disc}\label{lem150419a}
Consider the regular conductor square $\left( \square \right)$. 
Recall that an $A$-module $F$ is \emph{faithfully flat} if it is flat and for every $A$-module $N$ one has $N=0$ if and only if $F\otimes_AN=0$.
(See also~\cite[Theorem 7.2]{matsumura:crt}.)
Lemma~\ref{lem150510a} implies that if $T$ is faithfully flat over $R$, then $T=R$; see~\cite[p.~15]{glaz:ccr}.
\end{disc}

We  use  the following
conditions in Sections~\ref{sec150612a}--\ref{sec150612c}.
Note that  (U2)$\implies$(U1).

\begin{defn}\label{defn150522aa}
Consider a regular conductor square $\left( \square \right)$ and 
an ideal $I$ of $R$.
\begin{description}
\item[(FP)]
$T$ is flat over $R$ and $C$ is a principal ideal of $T$.
\item[(U1)]
There is an ideal $I'$ of $R$ isomorphic to $I$ such that $I'T=T$.
\item[(U2)]
There are elements $r,s\in R$ and an ideal $U\subseteq R$ such that $UT=T=rT$ and
$I\xra[\cong] r rI=sU\xla[\cong] sU$.\end{description}
In conditions (U1) and (U2), the ideals $I'$ and $U$ are ``unitary''.
\end{defn}

Next, we discuss  important cases where the conditions from Definition~\ref{defn150522aa} hold. 

\begin{disc}\label{disc150513a}
Assume that  $T$ is a PID and a localization of $R$.
Then conditions (FP), (U1), and (U2) from Definition~\ref{defn150522aa}  hold for all finitely generated ideals $I$ of $R$;
argue as in~\cite[Lemma 2.2]{mcquillan}.
For instance, let $K$ be a field with algebraic closure $\ol K$, and set $T=K[x]$. 
Let $\theta_1,\ldots,\theta_r\in \ol K$ with minimal polynomials $p_1,\ldots,p_r\in T$.
Assume that the $p_i$ are pairwise relatively prime.
For $i=1,\ldots,r$, let $A_i$ be a domain with field of fractions $K[\theta_i]$.

Set $A:=A_1\times\cdots\times A_r$ and $B:=K[\theta_1]\times\cdots\times K[\theta_r]$,
let $\eta_1\colon T\to B$ be the natural surjection $f\mapsto (f(\theta_1),\ldots,f(\theta_r))$, and let $\iota_1\colon A\into B$ be the natural inclusion.
Consider the  conductor square  determined by this data:
\begin{gather}
\begin{split}
\xymatrix{
R\ \ar@{^(->}[r]^-{\iota_2} \ar@{->>}[d]_-{\eta_2}
& K[X]\ar@{->>}[d]^-{\eta_1} \\ 
A_1\times\cdots\times A_r \ \ar@{^(->}[r]^-{\iota_1} &K[\theta_1]\times\cdots\times K[\theta_r].}
\end{split}
\tag{$\boxtimes$}
\end{gather}
Note that $K[X]$ is a localization of $R$.
Also, not that if $E$ is a finite subset of a domain $D$, then $\operatorname{Int}(E,D)$ is a special case of this construction.
\end{disc}

For the next result, recall that a ring $R$ is \emph{B\'ezout} if every finitely generated ideal of $R$ is principal. 

\begin{prop}\label{prop150614a}
In the regular conductor square $\left( \square \right)$,  assume that $T$ is flat over $R$.
Consider the following conditions.
\begin{enumerate}[\rm (i)]
\item \label{prop150614a12}
Every non-zero finitely generated ideal of $R$ satisfies the condition (U2).
\item \label{prop150614a22}
Every non-zero 2-generated ideal of $R$ satisfies the condition (U2).
\item \label{prop150614a11}
Every non-zero finitely generated ideal of $R$ satisfies the condition (U1).
\item \label{prop150614a21}
Every non-zero 2-generated ideal of $R$ satisfies the condition (U1).
\item \label{prop150614a3}
The ring $T$ is a B\'ezout domain. 
\end{enumerate}
The implications \eqref{prop150614a12}$\implies$\eqref{prop150614a22}$\implies$\eqref{prop150614a21}$\implies$\eqref{prop150614a3}
and \eqref{prop150614a12}$\implies$\eqref{prop150614a11}$\implies$\eqref{prop150614a21}$\implies$\eqref{prop150614a3} always hold.
If $T$ is  a localization of $R$, then the conditions~\eqref{prop150614a12}--\eqref{prop150614a3} are equivalent.
\end{prop}

\begin{proof}
The implications \eqref{prop150614a12}$\implies$\eqref{prop150614a22}$\implies$\eqref{prop150614a21}
and \eqref{prop150614a12}$\implies$\eqref{prop150614a11}$\implies$\eqref{prop150614a21} are routine.

\eqref{prop150614a21}$\implies$\eqref{prop150614a3}.
Assume that every non-zero 2-generated ideal of $R$ satisfies the condition (U1).
To show that $T$ is B\'ezout, it suffices to show that every 2-generated ideal of $T$ is principal; that is is sufficient follows from an induction argument
on the number of generators of a given finitely generated ideal. 
Let $t,u\in T$ and consider the ideal $J:=(t,u)T$. 
Let $c\in C$ be $T$-regular, noting that we have $ct,cu\in R$. 
It follows that $J\cong cJ=(ct,cu)T$, so we may replace $J$ with $cJ$ to assume without loss of generality that $t,u\in R$.
Our assumption implies that the ideal $I$ is isomorphic to an ideal $I'$ of $R$ such that $I'T=T$.
Since $T$ is flat over $R$, it follows that 
$$J=IT\cong \Otimes TI\cong\Otimes T{I'}\cong I'T=T$$
so $J$ is principal.
Note that this argument also shows that $T$ is a domain. Indeed, if $0\neq t=u$, then the argument above shows that $tT\cong T$, so $t$ is a non-zero-divisor on $T$.

It remains to assume that $T$ is a localization of $R$,  and to prove the implication~\eqref{prop150614a3}$\implies$\eqref{prop150614a12}.
Assume that $T=S^{-1}R$ is a B\'ezout domain, and let $I$ be a non-zero finitely generated ideal of $R$, say $I=(r_1,\ldots,r_n)R$.
We need to show that $I$ satisfies condition (U2)  from Definition~\ref{defn150522aa}. 
The extension $IT$ is finitely generated over $T$, hence it is principal, say $IT=(r_1,\ldots,r_n)T=tT$ with $t\in T$. 
The condition $I\neq 0$ implies that $tT=IT\neq 0$, since $R\subseteq T$, so we have $t\neq 0$.
Write $t=r/s$ for some $r\in R\ssm\{0\}$ and $s\in S$. 
Since $s$ is a unit in $T$, it follows that $IT=rT$. 
For $i=1,\ldots,n$, write $r_i=rt_i$ for some $t_i\in T=S^{-1}R$, and write $t_i=\rho_i/\sigma_i$ for some $\rho_i\in R$ and $\sigma_i\in S$.
Set $\sigma=\sigma_1\cdots\sigma_n$ and $\sigma_i'=\sigma/\sigma_i$, and note that $\sigma\in S\subseteq R$ satisfies $\sigma T=T$. 
Also, we have $\rho_i,\sigma_i'\in R$, so we set $U:=(\rho_1\sigma_i',\ldots,\rho_n\sigma_n')R$. 
It is straightforward to show that we have $I\xra[\cong]\sigma \sigma I=rU\xla[\cong]{r} U$ and $UT=T$, so condition (U1) is satisfied.
\end{proof}

The next result shows that many of our results are trivially true when the conductor square $\left(\square\right)$ is trivial, even if it is not regular.

\begin{prop}\label{prop150615a}
Consider the trivial conductor square $\left(\square\right)$. If every non-zero 2-generated ideal $I$ of $R$ satisfies condition (U1) from Definition~\ref{defn150522aa},
then $R=T$ is a B\'ezout domain.
\end{prop}

\begin{proof}
The condition $R=T$ follows from the triviality assumption.

We now show that $R$ is a B\'ezout ring.
Let $I$ be a non-zero finitely generated ideal. By assumption, there is an ideal $I'$ of $R$ such that $I\cong I'$ and $I'T=T$. 
The condition $R=T$ implies that $I\cong I'=I'R=I'T=T=R$, so $I$ is principal, as desired. 

This also shows that $R$ is a domain, since any principal ideal $I=rR$ satisfies $rR=I\cong R$, so $r$ is a non-zero-divisor.
\end{proof}

The next two lemmas document some technical facts for use in the sequel.

\begin{lem}\label{lem150517a}
Consider the regular conductor square $\left( \square \right)$. Assume that $T$ is flat as an $R$-module and that $C$ is a principal ideal of $T$. 
Let $I$ be an ideal of $R$.
\begin{enumerate}[\rm(a)]
\item\label{lem150517a1}
If $IT=T$, then $I\supseteq IC=C$ and $IA=I/C=I/CI\cong\Otimes AI$.
\item\label{lem150517a2}
The natural map $\mu\colon\Otimes CI\to I$ is injective and $\Tor 1AI=0$.
\item\label{lem150517a3}
If $IT$ is a principal regular ideal of $T$, e.g., if $IT=T$, 
then for all $i$ we have
$$\Ext iIC\cong\Ext iIT
\cong
\begin{cases}T&\text{if $i= 0$} \\ 0 & \text{if  $i\neq 0$.}\end{cases}
$$
\end{enumerate}
\end{lem}

\begin{proof}
\eqref{lem150517a1}
If $C/IC\cong T/IT=0$, then $C=IC\subseteq I$.
Alternately, if $T=IT$, then $C=TC=ITC=IC$.

\eqref{lem150517a2}
By assumption, we have $C=cT\cong T$ for some element $c\in C$.
Consider the following commutative diagram:
$$\xymatrix{
\Otimes TI\ar[r]^-{\cong}\ar[d]_\cong^{c\cdot} &TI_{}\ar@{^(->}[d]^{c\cdot}\\
\Otimes CI\ar[r]^-\mu&I
}$$
The horizontal maps are the natural ones: $t\otimes k\mapsto tk$.
It follows that $\mu$ is injective.

Now, consider the exact sequence
$$0\to C\to R\to A\to 0$$
and the induced long exact sequence in $\Tor {}-I$:
$$0\to\Tor 1AI\to\Otimes CI\xra{\mu} I\to I/CI\to 0.$$
Since $\mu$ is injective, it follows that $\Tor 1AI=0$, as desired.

\eqref{lem150517a3}
Since $C=cT\cong T$, the isomorphism $\Ext iIC\cong\Ext iIT$ is automatic. 
For the other isomorphism, let $P$ be a projective resolution of $I$ over $R$. 
The isomorphisms in the next sequence are Hom-tensor adjointness and Hom-cancellation:
\begin{align}
\Hom[T]{\Otimes TP}{T}
&\cong\Hom{P}{\Hom[T] TT}
\cong\Hom PT.
\label{eq150531a}
\end{align}
Since $T$ is flat over $R$, the complex $\Otimes TP$ is a projective resolution of $\Otimes TI\cong IT\cong T$;
the second isomorphism is from the assumption that $IT$ is principal and regular over $T$.
This explains the third isomorphism in the next sequence:
\begin{align*}
\Ext iIT
&\cong\HH^i(\Hom PT)\\
&\cong\HH^i(\Hom[T]{\Otimes TP}{T})\\
&\cong\Ext[T] iTT\\
&\cong
\begin{cases}T&\text{if $i= 0$} \\ 0 & \text{if  $i\neq 0$.}\end{cases}
\end{align*}
The first isomorphism is by definition,  the second one is from~\eqref{eq150531a}, and the last one is standard.
\end{proof}

\begin{lem}\label{lem150602a}
Consider the regular conductor square $\left( \square \right)$. Assume that $T$ is flat as an $R$-module and that $C$ is a principal ideal of $T$. 
Let $I$ be an ideal of $R$ such that  $IT=T$.
Then the natural maps 
$$C\to C\cdot\Hom IR\to C\cdot\Hom IT\to\Hom IC$$
are isomorphisms, and one has
$$\Hom IR/[C\cdot\Hom IR]\cong\Hom[A]{IA}{A}.$$ 
\end{lem}

\begin{proof}
Assume that $IT=T$. We first describe the ``natural maps'' from the statement.
For each $r\in R$, let $\lambda_r\colon I\to R$ be given by $\lambda_r(i)=ri$.
In particular, the map $\lambda_1\colon I\to R$ is the inclusion of $I$ in $R$.
Moreover, for all $r,s\in R$ we have $r\lambda_s=\lambda_{rs}$ and so $\lambda_r=r\lambda_1$.
In particular, the map $\Lambda\colon C\to C\cdot\Hom IR$ given by $x\mapsto x\lambda_1=\lambda_x$ is a well-defined $R$-module homomorphism. 
Furthermore, this map is a monomorphism, as follows. Suppose that $x\in\Ker(\Lambda)$, that is, that $\lambda_x=0$. 
By definition, it follows that $xI=0$, and this implies that $0=xIT=xT$. Since $T$ contains $1$, it follows that $x=0$.

Consider the inclusion $R\xra[\subseteq]{\epsilon} T$, and apply the left-exact functor $\Hom I-$ to obtain the monomorphism
$\Hom IR\xra{\epsilon_*}\Hom IT$. Note that $\epsilon_*$ simply enlarges the codomain of a homomorphism $I\to R$ from $R$ to the larger ring $T$. 
The natural map 
$C\cdot\Hom IR\xra\Phi C\cdot\Hom IT$ is induced by $\epsilon_*$:
for elements $c_j\in C$ and $f_j\in\Hom iR$, we have
$\Phi(\sum_jc_jf_j):=\sum_jc_j\epsilon_*(f_j)=\sum_jc_j(\epsilon\circ f_j)$.
Note that this gives a commutative diagram
$$\xymatrix{
C\cdot\Hom IR
\ar[r]^{\Phi}\ar@{^(->}[d]
&C\cdot\Hom IT
\ar@{^(->}[d] \\
\Hom IR\ 
\ar@{^(->}[r]^{\epsilon_*}
&\Hom IT}$$
where the vertical monomorphism are the subset inclusions.
Since $\epsilon_*$ is injective, the commutative diagram shows that $\Phi$ is also injective.

Next, consider elements $c_j\in C$ and $g_j\in\Hom IT$.
Note that the function $g:=\sum_jc_jg_j$ satisfies $g(i)=\sum_jc_jg_j(i)\in CT=C$,
that is, we have $\im(g)\subseteq C$;
let $g'\colon I\to C$ denote the homomorphism obtained by restricting the codomain of $g$ from $T$ to $C$.
We define $C\cdot\Hom IT\xra\Psi\Hom IC$
by the formula $\Psi(g):=g'$.
Note that if $\rho\colon C\to T$ denotes the inclusion map and $\rho_*\colon\Hom IC\to\Hom IT$ is the induced monomorphism,
then our definition implies that we have $\rho_*(\Psi(g))=\rho_*(g')=\rho\circ g'=g$.
Also, note that $g$ and $g'$ use the same rule of assignment: for all $i\in I$, we have $g(i)=g'(i)$, by definition.
It follows that $g=0$ if and only if $g'=0$, that is, if and only if $\Psi(g)=0$; thus, $\Psi$ is injective. 

Since the maps $\Lambda$, $\Phi$, and $\Psi$ are monomorphisms, to prove that they are isomorphisms, it suffices to show that the composition
$\Psi\circ\Phi\circ\Lambda$ is surjective. 
It is straightforward to show that, for all $x\in C$ and all $i\in I$, we have $\Psi(\Phi(\Lambda(x)))(i)=xi$. 

On the other hand, we consider the following sequence:
\begin{align*}
T
&\cong\Hom[T] TT \\
&=\Hom[T]{IT}T \\
&\cong\Hom[T]{\Otimes TI}T \\
&\cong\Hom{I}{\Hom[T]TT} \\
&\cong\Hom IT.
\end{align*}
The first and last steps are Hom-cancellation.
The second step is from the assumption $IT=T$.
The third step is by the flatness of $T$ over $R$, which implies that the map $\Otimes TI\to IT$, given by $t\otimes i\mapsto it$, is an isomorphism.
The fourth step is Hom-tensor adjointness.
It is routine to show that the composition of these isomorphisms $\alpha\colon T\to\Hom IT$ is given by the formula $\alpha(t)(i)=it$. 

Multiplying by $C$ yields an isomorphism $C=CT\xra[\cong]{\alpha'}C\cdot\Hom IT$ which is given by
$\alpha'(x)=\alpha'(x\cdot 1)=x\alpha(1)=\alpha(x)$. In other words, we have
$\alpha'(x)(i)=xi=\Psi(\Phi(\Lambda(x)))(i)$, so $\alpha'(x)=\Psi(\Phi(\Lambda(x)))$ and hence $\alpha'=\Psi\circ\Phi\circ\Lambda$.
Since $\alpha'$ is an isomorphism, it follows that $\Psi\circ\Phi\circ\Lambda$ is surjective, as claimed. 
This shows that $\Lambda$, $\Phi$, and $\Psi$ are isomorphisms.

To complete the proof, consider the exact sequence 
$$0\to C\to R\to A\to 0.$$
Since $\Ext 1IC=1$ by Lemma~\ref{lem150517a}\eqref{lem150517a3},
the induced long exact sequence in $\Ext{}I-$ begins as follows:
$$0\to \Hom IC\to\Hom IR\to\Hom IA\to 0.$$
By what we have already shown, $\Hom IC$ is naturally identified with the submodule $C\cdot\Hom IR\subseteq\Hom IR$, so this sequence has the form
$$0\to C\cdot\Hom IR\xra\subseteq\Hom IR\to\Hom IA\to 0.$$
Thus, we have the first isomorphism in the next sequence:
\begin{align*}
\Hom IR/[C\cdot\Hom IR]
&\cong \Hom IA \\
&\cong \Hom I{\Hom[A]AA} \\
&\cong\Hom[A]{\Otimes AI}A \\
&\cong\Hom[A]{IA}A.
\end{align*}
The second and third isomorphisms are cancellation and adjointness.
And the fourth isomorphism is from
Lemma~\ref{lem150517a}\eqref{lem150517a1}.
This sequence of isomorphisms completes the proof.
\end{proof}

\section{$C$ is not finitely generated over $R$}\label{sec150419a}

The point of this section is to prove Theorem~\ref{prop150419a}, which says that, given a non-trivial, flat, regular conductor square $\left( \square \right)$,
the conductor ideal $C$ is never finitely generated over $R$, and $B$ is never faithfully flat over $A$.
Note, however, that $C$ can be finitely generated over $T$ (indeed, it can be principal), and $B$ is always flat over $A$ in this setting.

\begin{lem}\label{lem150316a}
Consider the regular conductor square $\left( \square \right)$, and assume that $C$ is finitely generated over $R$ (e.g., $C$ 
satisfies $(\dagger)$ over $R$). 
\begin{enumerate}[\rm(a)]
\item\label{lem150316a1}
For each prime ideal $\fp\subset A$, there is a prime ideal $\fq\subset B$ lying over  $\fp$. 
\item\label{lem150316a2}
For each maximal ideal $\fm\subset A$, there is a maximal ideal $\fn\subset B$ lying over $\m$. 
\item\label{lem150316a3}
If $T$ is flat as an $R$-module, then $B$ is faithfully flat as an $A$-module.
\end{enumerate}
\end{lem}

\begin{proof}
\eqref{lem150316a1}
Let $P\subset R$ be the contraction of $\fp$ along the surjection $R\onto A$.
The localized square
\begin{gather}
\begin{split}
\xymatrix{
R_{P}\ \ar@{^(->}[r] \ar@{->>}[d]
& T_{P}\ar@{->>}[d] \\ 
A_{P} \ \ar@{^(->}[r] &B_{P}}
\end{split}
\tag{$\square_{P}$}
\end{gather}
is a regular conductor square with conductor $C_{P}$; see~\cite[Notation 2.4]{boynton:rpb}.
In particular, the ring $B_{P}$ is non-zero, so it has a maximal ideal, necessarily of the form $\fq_{P}$ by the prime correspondence under localization.
The contraction in $A_{P}=A_{\fp}$ of $\fq_{P}=\fq_{\fp}$ must be maximal in $A_{\fp}$ by~\cite[Proposition 2.5(iii)]{boynton:rpb};
this is where we use the finite generation of $C$.
Thus, the contraction in $A_{\fp}$ of $\fq_{\fp}$ must be $\fp_{\fp}$.
Another application of the prime correspondence implies that the contraction of $\fq$ in $A$ is $\fp$, as desired.

\eqref{lem150316a2}
Part~\eqref{lem150316a1} implies that there is a prime ideal $\fq$ of $B$ that contracts to $\fm$ in $A$.
If $\fq$ is not maximal, let $\fn\subset B$ be a maximal ideal containing $\fq$.
Then the contraction of $\fn$ in $A$ must be a prime ideal containing $\fm$, that is, the contraction must be $\fm$.

\eqref{lem150316a3}
Since $T$ is flat over $R$, we know that $B$ is flat over $A$. 
Thus, part~\eqref{lem150316a2} implies that $B$ is faithfully flat over $A$; see~\cite[Theorem 7.2]{matsumura:crt}.
\end{proof}

A special case of the next result is in~\cite[Lemma 4.1]{gabelli:ccp}; see also~\cite[Lemma 1]{brewerandrutter}.

\begin{thm}\label{prop150419a}
Consider the non-trivial regular conductor square $\left( \square \right)$. Assume that 
$T$ is flat as an $R$-module. 
Then $B$ is not faithfully flat over $A$, and $C$ is not finitely generated over $R$
\end{thm}
 
\begin{proof}
Suppose by way of contradiction that $B$ is  faithfully flat over $A$.
We derive a contradiction by showing that $T=R$. 
To accomplish this, by Remark~\ref{lem150419a} it suffices to show that $T$ is faithfully flat over $R$.
Since $T$ is flat over $R$, it suffices to show that for each maximal ideal $\m\subset R$ there is a maximal ideal $\n\subset T$ such that
$\m=R\cap\n$.
If $\m$ contains $C$, then this follows from Lemma~\ref{lem150316a}\eqref{lem150316a2}, via the prime correspondence under quotients.
On the other hand, if $\m$ does not contain $C$, then there is a unique prime ideal $\q$ of $T$ contracting to $\m$ in $R$, by~\cite[Lemma 1.1.4(3)]{FHP}.
As in the proof of Lemma~\ref{lem150316a}\eqref{lem150316a2}, this yields a maximal ideal of $T$ contracting to $\m$.
(Of course, the uniqueness of $\q$ implies that $\q$ itself is maximal.)

Lastly, if $C$ were finitely generated over $R$, then $B$ would be faithfully flat over $A$, by Lemma~\ref{lem150316a}\eqref{lem150316a3},
contradicting the previous paragraph. 
\end{proof}

We conclude this section by documenting some corollaries and examples.

\begin{cor}\label{cor150506a}
Consider the non-trivial regular conductor square $\left( \square \right)$. 
Assume that $C$ is maximal in $R$, that is, that $A$ is a field.  
Then $T$ is not flat over $R$.
\end{cor}

\begin{proof}
Since $A$ is a field, the extension $A\to B$ is faithfully flat. 
Now apply Theorem~\ref{prop150419a}.
\end{proof}

\begin{ex}\label{ex150612b}
Consider the regular pullback 
$$\xymatrix{
R=\bbq+x\bbr[x]\ \ar@{^(->}[r]^-{\iota_2} \ar@{->>}[d]_-{\eta_2}
& \bbr[x]\ar@{->>}[d]^-{\eta_1} \\ 
\bbq \ \ar@{^(->}[r]^-{\iota_1} &\bbr}
$$
where $\eta_1$ maps $x\mapsto 0$.
Then Corollary~\ref{cor150506a} shows that $\bbr[x]$ is not flat over $R$.
\end{ex}

\begin{cor}\label{150612a}
Consider the non-trivial regular conductor square $\left( \square \right)$. Assume that $R$ is a Pr\"ufer ring.
Then $C$ is not finitely generated over $R$
\end{cor}
 
\begin{proof}
Since $R$ is Pr\"ufer, every overring of $R$ is flat over $R$, including the overring $T$.
Now apply Theorem~\ref{prop150419a}.
\end{proof}

\begin{ex}\label{ex150612a}
Consider the regular pullback $(\boxtimes)$ with $r=1$, $A_1=\bbz$, $K=\bbq$, and $\theta_1=0$.
It is well known that the pullback $R=\bbz+x\bbq[x]$ is Pr\"ufer.
Thus, the conductor ideal $x\bbq[x]$ is not finitely generated over $R$.
\end{ex}

\begin{cor}\label{cor150506b}
Consider the non-trivial regular conductor square $\left( \square \right)$. 
Assume that $R$ is noetherian. Then $T$ is not flat over $R$.
\end{cor}

\begin{proof}
Apply Theorem~\ref{prop150419a}.
\end{proof}

\begin{ex}\label{ex150506a}
Assume that $A$ is a field with $B\neq A$, that $B$ is finitely generated as an $A$-module, and $T$ is noetherian.
Then  $R$ is noetherian by~\cite[Theorem 3.2]{boynton:rpb},
and $B$ is automatically flat over $A$, but each corollary shows that $T$ is never flat over $R$.
A specific example of this is the ring $R=\bbq+(x^2+1)\bbq[x]$,
arising from the data $A=\bbq$, $B=\bbq[i]$, and $T=\bbq[x]$.
\end{ex}

\section{Finite Conductor Rings and Coherent Rings}\label{sec150612a}

This section is devoted to the study of the transfer of the following two properties in flat, regular pullbacks.
Note that Gabelli and Houston~\cite{gabelli:ccp} investigate the special case where $C$ is maximal in $T$.

\begin{defn}\label{notn150509a}
We consider the following coherency conditions on a ring $R$.
\begin{enumerate}[(1)]
\item\label{notn150509a3}
\emph{finite conductor ring}: every 2-generated ideal of $R$ is finitely presented.
\item\label{notn150509a2}
\emph{coherent ring}: every finitely generated ideal of $R$ is finitely presented.
\end{enumerate}
\end{defn}

We investigate the stably coherent situation in~\cite{boynton:sncr}.

\begin{disc}\label{disc150509a}
Every B\'ezout domain is coherent, and every
coherent ring is a finite conductor ring.
In particular, in the case of trivial conductor squares, the results of this section follow from Proposition~\ref{prop150615a}.
\end{disc}

The coherent case of the next result is a special case of~\cite[Proposition 2.2]{greenberg:ccs}.

\begin{prop}\label{thm150506a}
Consider the regular conductor square $\left( \square \right)$. Assume that $T$ is flat as an $R$-module. 
If $R$ is a finite conductor ring (respectively, $R$ is coherent), then $T$ is as well.
\end{prop}

\begin{proof}
Assume first that $R$ is a finite conductor ring,
and let $J$ be a 2-generated ideal of $T$.
It suffices to show that $J$ is finitely presented over $T$.
By assumption, there is a $T$-regular element $c\in C$.
It follows that the ideal $cJ$ is isomorphic to $J$ (in particular, it is 2-generated) so it suffices to show that $cJ$ is finitely presented over $T$.
Furthermore, $cJ$ is generated over $T$ by two elements of $cJ\subseteq C\subseteq R$.
Thus, we replace $J$ with $cJ$ to assume that $J$ is generated by two elements $r,s\in R$.

Since $R$ is a finite conductor ring, the ideal $I=(r,s)R$ is finitely presented over $R$.
By construction, we have $IT=J$.
From right-exactness, the module $T\otimes_RI$ is finitely presented over $T$.
By flatness, we have $T\otimes_RI\cong IT=J$, that is, $J$ is finitely presented over $T$, as desired.

The coherent case is proved similarly.
\end{proof}

The next two results show how our properties of interest descend in a flat, regular pullback square.

\begin{thm}\label{thm150412a}
Assume that the regular conductor square $\left( \square \right)$ satisfies condition (FP),
and every non-zero 2-generated ideal $I$ of $R$ satisfies condition (U1) from Definition~\ref{defn150522aa}.
If $A$ is a finite conductor ring, then so are $R$ and $T$.
\end{thm}

\begin{proof}
Assume that $A$ is a finite conductor ring.
By Proposition~\ref{thm150506a}, it suffices to show that $R$ is also a finite conductor ring.
Let $I$ be a 2-generated ideal of $R$.
We need to show that $I$ is finitely presented over $R$.

By assumption, there is an ideal $I'\subseteq R$ isomorphic to $I$ such that $I'T=T$.
Thus, we may replace $I$ with $I'$ to assume that $IT=T$. 
Lemma~\ref{lem150517a}\eqref{lem150517a1} implies that $IC=C$.

To show that $I$ is finitely presented over $R$, it suffices by~\cite[Theorem 5.1.2]{glaz:ccr} 
to show that $T\otimes_RI$ is finitely presented over $T$
and that $I/CI=I/C$ is finitely presented over $R/C=A$. 
Note that this uses the fact that $C\cong T$ is flat over $R$ and that $CT=C$.

Since $T$ is flat over $R$, we have $T\otimes_RI\cong IT=T$, which
is finitely presented over $T$.
Also, the ideal $IA$ is 2-generated over $A$, so it is finitely presented over $A$, since $A$ is a finite conductor ring.
That is, the $A$-module $I/CI=I/C=IA$ is finitely presented, as desired.
\end{proof}

\begin{thm}\label{thm150412az}
Assume that the regular conductor square $\left( \square \right)$ satisfies condition (FP),
and every finitely generated non-zero ideal $I$ of $R$ satisfies condition (U1) from Definition~\ref{defn150522aa}.
If $A$ is coherent, then so are $R$ and $T$.
\end{thm}

\begin{proof}
This is similar to the proof of Theorem~\ref{thm150412a}.
\end{proof}

\begin{rmk}\label{disc150520a}
Note that Theorem~\ref{thm150412az} is similar in spirit to~\cite[Theorem 5.1.3]{glaz:ccr}:
with some extra assumptions, if $A$ and $T$ are coherent, then so is $R$.
However, our result covers some examples that~\cite[Theorem 5.1.3]{glaz:ccr} does not cover, and vice versa.
For instance, in the case $(\boxtimes)$ of~\ref{disc150513a}, if $r\geq 2$, and $A$ is non-noetherian with infinite weak dimension,
then Theorem~\ref{thm150412az} applies while~\cite[Theorem 5.1.3]{glaz:ccr} does not; see, e.g., Corollary~\ref{cor150517a}.

On the other hand, Theorem~\ref{thm150412az} is somewhat different from~\cite[Theorem 5.1.3]{glaz:ccr}, 
in that we only assume that $A$ is coherent. 
This is a byproduct of the assumption (U1) from Definition~\ref{defn150522aa}, in light of Proposition~\ref{prop150614a}. 
\end{rmk}

Our next results treat the ascent of our properties of interest.

\begin{thm}\label{thm150506c}
Assume that the regular conductor square $\left( \square \right)$ satisfies condition (FP)
and every non-zero 2-generated ideal $I$ of $R$ satisfies condition (U2) from Definition~\ref{defn150522aa}.
Assume further that $A$ is  locally a domain and has finite Krull dimension.
Then $R$ is a finite conductor ring if and only if  $A$ is a finite conductor ring; each of these conditions implies that $T$ is a finite conductor ring.
\end{thm}

\begin{proof}
In view of Proposition~\ref{thm150506a} and Theorem~\ref{thm150412az}, it remains to assume that $R$ is a finite conductor ring and prove that $A$ is so.
Let $J=(a,b)A$ be a 2-generated ideal of $A$. 
We need to prove that $J$ is finitely presented over $A$.
Assume without loss of generality that $J\neq 0$, and, further, that $a\neq 0$.
Let $f,g\in R$ be such that $\eta_2(f)=a$ and $\eta_2(g)=b$.
Set $I=(f,g)R$. 

Case 1: $IT=T$.
Since $I$ is 2-generated and $R$ is a finite conductor ring, the ideal $I$ is finitely presented over $R$.
It follows (e.g., by right-exactness of tensor-product) that the $A$-module $I/CI$ is finitely presented. 
Lemma~\ref{lem150517a}\eqref{lem150517a1} implies that $CI=C$, so the ideal 
$$I/C=IA=(f,g)A=(a,b)A=J$$ 
is finitely presented, as desired.

Case 2: 
there is an ideal $I'\subseteq R$ such that $I'T=T$ and an element $h\in I$ such that
the multiplication map $I'\xra h I$ is an isomorphism.
In particular, $I'$ is 2-generated,
by elements $f',g'\in I'$ such that $f=f'h$ and $g=g'h$.
Set $a'=\eta_2(f')$ and $b'=\eta_2(g')$ and $x=\eta_2(h)$.
It follows that we have 
$$a=\eta_2(f)=\eta_2(f')\eta_2(h)=a'x$$
and similarly $b=b'x$.
In particular, the condition $a\neq 0$ implies that $x\neq 0$.

Consider the natural surjection $\tau'\colon A^2\to J'$, 
represented by the row matrix $\begin{pmatrix}a'&b'\end{pmatrix}$.
By Case 1, the ideal $J'=(a',b')A$ is finitely presented over $A$,
so $K':=\Ker(\tau')$ is $n$-generated for some integer $n$.
Assume without loss of generality that $n\geq 2$.

To show that $J$ is finitely presented, consider the natural surjection $\tau\colon A^2\to J$, 
represented by the row matrix $\begin{pmatrix}a&b\end{pmatrix}$.
We need to show that $K:=\Ker(\tau)$ is finitely generated. 
We accomplish this using~\cite[Theorem 2.1]{wiegand:dfps}, which says that it suffices to show that 
$K_{\p}$ is $n$-generated over $A_\p$ for each prime $\p\subset A$.
(Here is where we use the finiteness of $\dim(A)$. In the language of~\cite{wiegand:dfps}, this allows us to conclude that the function
$$b(\p,M):=\begin{cases}\mu(A_\p,M_\p)+\dim(A/\p)&\text{if $\p\in\Supp_A(K)$} \\ 0&\text{otherwise}\end{cases}$$
satisfies $b(\p,M)\leq n+\dim(A)$ for all $\p$.)

Consider the multiplication map $J'\xra x J$, which is surjective (hence, locally surjective) by construction.
Fix a prime $\p\subset A$. 
If the map $J'_\p\xra x J_\p$ is injective, then it is an isomorphism (since locally surjective). 
In this situation, given the defining matrices for $\tau$ and $\tau'$, it follows that $K_\p\cong K'_\p$,
so $K_\p$ is also $n$-generated. 

Thus, we assume for the remainder of the proof of Case 2 that the map $J'_\p\xra x J_\p$ is not injective.

Claim: $J_\p=0$. 
Let $z\in J'$ and $s\in A\ssm \p$ be such that the element $z/t$ in $J'_\p$  is non-zero and satisfies $0=(z/t)x=(z/t)(x/1)$.
Since $A_\p$ is a domain by assumption, it follows that $x/1=0$ in $A_\p$.
Thus, we have $J_\p=(x/1)J'_\p=0$, as claimed. 

From the claim, it follows that $K_\p\cong A_\p^2$, which is 2-generated.
Thus, it is $n$-generated, since $n\geq 2$, as desired.

Case 3: the general case.
By assumption, 
there are elements $r,s\in R$ and an ideal $U\subseteq R$ such that $UT=T=rT$ and
$I\xra[\cong] r rI=sU\xla[\cong] sU$.
Case 2 shows that $rIA=sUA$ is finitely presented.
As in the proof of Case 2, the map $J=IA\xra rrIA$ is surjective.
The assumption $rT=T$ says that $r$ is a unit in $T$,
so $r$ represents a unit in $T/C=B$. 
In particular, multiplication by $r$, restricted to the subset $IA\subseteq A\subseteq B$ is injective,
so we have $J=IA\xra[\cong] rrIA$. Since $rIA$ is finitely presented over $A$, so is $J$, as desired.
\end{proof}

The next  result is proved like the previous one.
Compare to~\cite[5.1.3]{glaz:ccr}.

\begin{thm}\label{thm150506cz}
Assume that the regular conductor square $\left( \square \right)$ satisfies condition (FP)
and every non-zero finitely generated ideal $I$ of $R$ satisfies condition (U2) from Definition~\ref{defn150522aa}.
Assume further that $A$ is  locally a domain and has finite Krull dimension.
Then $R$ is coherent if and only if $A$ is coherent; each of these conditions implies that $T$ is coherent.
\end{thm}

\begin{thm}\label{thm150507a}
Assume that the regular conductor square $\left( \square \right)$ satisfies condition (FP)
and every non-zero 2-generated ideal $I$ of $R$ satisfies condition (U2) from Definition~\ref{defn150522aa}.
Assume further that $A$ is a (possibly infinite) product $\prod_{\lambda\in\Lambda}A_\lambda$ of domains.
Then $R$ is a finite conductor ring if and only if $A$ is a finite conductor ring;
each of these conditions implies that $T$ is a finite conductor ring, and so is each ring $A_\lambda$.
\end{thm}

\begin{proof}
As in the proof of Theorem~\ref{thm150506c}, we assume that $R$ is a finite conductor ring, and we prove
that $A$ and $A_\lambda$ are finite conductor rings.
For $A$, let $J=(a,b)A$ with $a=(a_\lambda)\in \prod_{\lambda\in\Lambda}A_\lambda=A$ and $b=(b_\lambda)$.
To show that $J$ is finitely presented over $A$, we consider two cases.

Case 1: for each index $\lambda$, either $a_\lambda\neq 0$ or $b_\lambda\neq 0$.
Let $f,g\in R$ be such that $\eta_2(f)=a$ and $\eta_2(g)=b$, and set $I=(f,g)R$.
Since $R$ is a finite conductor ring, the ideal $I$ is finitely presented over $R$.
It follows (e.g., by right-exactness of tensor-product) that the quotient $I/CI$ is finitely presented over $A$. 

Sub-case 1a: $IT=T$. Then Lemma~\ref{lem150517a}\eqref{lem150517a1} implies that
$CI=C$, so the ideal 
$I/C=IA=J$
is finitely presented over $A$, by the previous paragraph. 

Sub-case 1b:  $I$ is isomorphic to an ideal $I'$ such that $I'T=T$ via a multiplication map
$I'\xra[\cong]h I$  for a fixed $h\in R$.
In particular, $I'$ is 2-generated,
by elements $f',g'\in I'$ such that $f=f'h$ and $g=g'h$.
Set $a'=\eta_2(f')$ and $b'=\eta_2(g)$ and $x=\eta_2(h)$.
Write $a'=(a'_\lambda)$ and similarly for $b'$ and $x$.
As in the proof of Theorem~\ref{thm150506c},
we have $a=xa'$, hence $a_\lambda=x_\lambda a'_\lambda$ and similarly for $b_\la$, for all $\lambda\in\Lambda$.
In particular, since $a_\lambda\neq 0$ or $b_\la\neq 0$ for each $\la$, we have $x_\lambda\neq 0$.

Since each ring $A_\lambda$ is a domain, and each coordinate of $x$ is non-zero,
the map $J':=(a',b')A\xra x(a,b)A=:J$ 
is injective;
it is surjective by construction. 
Since $I'T=T$, Sub-case 1a implies that $J'$ is finitely presented. 
Hence, the ideal $J\cong J'$ is finitely presented. This completes Sub-case 1b.

Now, we complete the proof in Case 1.
By assumption, there are elements $r,s\in R$ and an ideal $U\subseteq R$ such that $UT=T=rT$ and
$I\xra[\cong] r rI=sU\xla[\cong] sU$. Sub-case 1b implies that $rIA=sUA$ is finitely presented over $A$.
As in the proof of Theorem~\ref{thm150506c}, the map $J=IA\xra r rIA$ is an isomorphism, so $J$ is finitely presented as well.
This completes the proof in Case 1.

Case 2: the general case.
For each $\la\in\Lambda$, set
$$\wti a_\la:=\begin{cases}a_\la & \text{if $a_\la\neq 0$ or $b_\la\neq 0$} \\ 1 & \text{if $a_\la=0=b_\la$.} \end{cases}$$
Set $\wti a=(\wti a_\la)$ and $\wti J:=(\wti a,b)A$. Note that $\wti J$ satisfies the hypotheses of Case 1, so it is finitely presented.

For each  $\la\in\Lambda$, set $J_\la:=(a_\la,b_\la)A_\la$ and $\wti J_\la:=(\wti a_\la,b_\la)A_\la$.
Note that we have
$$\wti J_\la= \begin{cases}J_\la & \text{if $J_\la\neq 0$} \\ A_\la & \text{if $J_\la=0$.} \end{cases}$$
From this, it is straightforward to show that $J$ is a direct summand of $\wti J$.
(Specifically, for each $\la\in\Lambda$, set
$$\ol J_\la= \begin{cases}0 & \text{if $J_\la\neq 0$} \\ A_\la & \text{if $J_\la=0$.} \end{cases}$$
Then one has $\wti J\cong J\bigoplus\ol J$.)
Since $\wti J$ is finitely presented over $A$, it is straightforward to show that each summand (in particular, the summand $J$) is finitely presented over $A$.
This completes Case 2, so we conclude that $A$ is a finite conductor ring.

Fix an index $\mu\in\Lambda$.
To show that $A_\mu$ is a finite conductor ring, let $J_\mu=(a_\mu,b_\mu)A_\mu$ be a 2-generated ideal of $A_\mu$.
For all $\lambda\in\Lambda\ssm\{\la\}$, set $a_\la=0=b_\la$ and $J_\la=0$.
Also, set $a=(a_\lambda)\in \prod_{\lambda\in\Lambda}A_\lambda=A$ and $b=(b_\lambda)$.
Since $A$ is a finite conductor ring, the ideal $(a,b)J=\prod_{\la\in\Lambda}J_\la$ is finitely presented over $A=\prod_{\la\in\Lambda}A_\la$.
It follows readily that $J_\mu$ is finitely presented over $A_\mu$, as desired. 
\end{proof}

The next result is proved like the previous one.

\begin{thm}\label{thm150507az}
Assume that the regular conductor square $\left( \square \right)$ satisfies condition (FP)
and every non-zero finitely generated ideal $I$ of $R$ satisfies condition (U2) from Definition~\ref{defn150522aa}.
Assume further that $A$ is a (possibly infinite) product $\prod_{\lambda\in\Lambda}A_\lambda$ of domains.
Then $R$ is coherent if and only if $A$ is coherent;
each of these conditions implies that $T$ is coherent, and so is each ring $A_\lambda$.
\end{thm}

\begin{disc}\label{rmk150508a}
Let $A$ be a (possibly infinite) product $\prod_{\lambda\in\Lambda}A_\lambda$ of non-zero commutative rings with identity.
The last paragraph of the proof of Theorem~\ref{thm150507a} shows that, if $A$ is a finite conductor ring,
then the same is true of each factor of $A$. This implies that any sub-product $\prod_{\lambda\in\Lambda'}A_\lambda$ 
with $\Lambda'\subseteq\Lambda$ is also a  finite conductor ring.
Conversely, if $\Lambda$ is finite, and each $A_\la$ is a finite conductor ring, then so is the finite product $A$.
Example~\ref{ex150508a} shows that this converse fails when $\Lambda$ is infinite. 
We deduce that, when $\Lambda$ is infinite, the conclusion of Theorem~\ref{thm150507a} is very strong.
(Similar comments hold for the other classes of rings considered below.)
\end{disc}

\begin{ex}\label{ex150508a}
Let $k$ be a field.
For each integer $n\geq 2$, consider the polynomial ring
$S_n:=k[X_1,\ldots,X_n,Y_1,\ldots,Y_n]$.
Let $\fa_n\subseteq S_n$ denote the ideal generated by the $2\times 2$ minors of the matrix
$\left(\begin{smallmatrix}X_1&\ldots&X_n\\ Y_1&\ldots&Y_n\end{smallmatrix}\right)$,
and set $A_n:=S_n/\fa_n$.
Consider the 2-generated ideal $I_n:=(X_1,Y_1)A_n$ and the natural surjection $\tau_n\colon A_n^2\to I_n$ with kernel $K_n$.
It is straightforward to show that $K_n$ contains the following vectors:
$\left(\begin{smallmatrix}Y_1\\ -X_1\end{smallmatrix}\right),
\left(\begin{smallmatrix}Y_2\\ -X_2\end{smallmatrix}\right),\ldots,
\left(\begin{smallmatrix}Y_n\\ -X_n\end{smallmatrix}\right)$.
Moreover, since $\tau_n$ is minimal, and the entries of these vectors are homogeneous and linear,
we conclude that these vectors are minimal generators of $K_n$. 
In particular, since $A_n$ and $I_n$ are graded, we conclude that each $K_n$ requires at least $n$ generators.

Now, set $A:=\prod_{n=2}^\infty A_n$, and consider the ideal $I:=\prod_{n=2}^\infty I_n\subseteq A$.
Since each $I_n$ is 2-generated over $A_n$, the ideal $I$ is 2-generated over $A$.
The product $\tau\colon A^2\to I$ of the maps $\tau_n$ has kernel $K=\prod_{n=2}^\infty K_n$.
Since each $K_n$ requires at least $n$ generators, it follows that $K$ is not finitely generated. 
Thus, even though each $A_n$ is noetherian (hence coherent and a finite conductor ring),
the product $A$ is neither coherent nor a finite conductor ring.
\end{ex}

\begin{cor}\label{cor150517a}
Consider the regular conductor square $(\boxtimes)$ from Remark~\ref{disc150513a}.
Then  $R$ is a finite conductor ring (resp., coherent)
if and only if each $A_i$ is so.
\end{cor}

\begin{proof}
Remark~\ref{disc150513a} says that the hypotheses of 
Theorems~\ref{thm150507a}--\ref{thm150507az} are satisfied.
\end{proof}

\begin{cor}\label{cor150517b}
Let $D$ be a domain, and let $E\subseteq D$ be a finite subset.
Then the ring of integer-valued polynomials
$\operatorname{Int}(E,D)$ is a finite conductor ring (respectively, coherent)
if and only if $D$ is so.
\end{cor}

\begin{proof}
We have $R=\operatorname{Int}(E,D)$ in the following special case of the conductor square $(\boxtimes)$ with $r=|E|$:
$$\xymatrix{
R\ \ar@{^(->}[r]^-{\iota_2} \ar@{->>}[d]_-{\eta_2}
& K[X]\ar@{->>}[d]^-{\eta_1} \\ 
D^r \ \ar@{^(->}[r]^-{\iota_1} &K^r.}
$$
Thus, the desired conclusion follows from Corollary~\ref{cor150517a}.
\end{proof}

\section{GCD Domains and Generalizaed GCD Rings}
\label{sec150612b}

Next, we turn our attention to transfer of the following two GCD properties.

\begin{defn}\label{notn150509a'}
We consider the following coherency conditions on a ring $R$.
\begin{enumerate}[(1)]
\item\label{notn150509a4}
\emph{generalized GCD ring}: 
every principal ideal of $R$ is projective, and  every intersection of
two finitely generated flat ideals is finitely generated and flat.
\item\label{notn150509a4'}
\emph{GCD domain}: $R$  is a domain such that every intersection of
two principal ideals is principal.
\end{enumerate}
\end{defn}

\begin{rmk}\label{disc150507az}
Let $R$ be a domain. It is straightforward to show that $R$ is a GCD domain if for every 2-generated ideal $(r,s)R$ with $r,s\neq 0$ the kernel of the natural map
$R^2\to(r,s)R$ is cyclic; indeed, the kernel of this map is isomorphic to $rR\bigcap sR$.
Note that whenever this kernel is cyclic, it is isomorphic to $R$, since it is isomorphic to a non-zero principal ideal in the domain $R$.

In particular, this shows that every B\'ezout domain is a GCD domain, and every GCD domain is a finite conductor domain.
\end{rmk}

\begin{disc}\label{disc150509a'}
A result of Glaz~\cite{glaz:fcp} says that $R$ is a generalized GCD ring if and only if it is a finite conductor ring and locally a GCD domain.
Thus, every GCD domain is a generalized GCD ring, by Remark~\ref{disc150507az}.
In particular, in the case of trivial conductor squares, the results of this section follow from Proposition~\ref{prop150615a}.
\end{disc}

Our transfer results for this context begin with the following ascent result.

\begin{prop}\label{thm150506a'}
Consider the regular conductor square $\left( \square \right)$. Assume that $T$ is flat as an $R$-module. 
If $R$ is a GCD domain (respectively, a generalized GCD ring), then $T$ is as well.
\end{prop}

\begin{proof}
Assume that $R$ is a  GCD domain.
Remark~\ref{lem150509a} implies that $T$ is a domain.
Let $t,u\in T\ssm\{0\}$ and consider the intersection $tT\bigcap uT$.
Let $c\in C$ be a $T$-regular element.
It is straightforward to show that $ctT\bigcap cuT=c(tT\bigcap uT)\cong tT\bigcap uT$.
Thus, to show that $tT\bigcap uT$ is principal, it suffices to show that $ctT\bigcap cuT$ is principal.
Hence, we assume without loss of generality that $t,u\in C\subseteq R$.
Since $R$ is a GCD domain, we have $tR\bigcap uR=rR$ for some $r\in R$.
The flatness of $T$ over $R$ implies that
$$tT\bigcap uT=tRT\bigcap uRT=(tR\bigcap uR)T=rRT=rT$$
as desired.

Next, assume that $R$ is a generalized GCD ring, i.e., a finite conductor ring and locally a GCD domain; see Remark~\ref{disc150509a'}.
Consider a prime ideal $Q\subset T$, and set $P:=R\bigcap Q$.
Then $T_Q$ is a localization of the ring $T_P$.
Since $R_P$ is a GCD domain, the localized square $(\square_P)$ from Lemma~\ref{lem150316a}
shows that $T_P$ is a GCD domain, by the previous paragraph.
It follows that $T_Q$ is a GCD domain as well.
Also, since
$R$ is a finite conductor ring,
Proposition~\eqref{thm150506a}
implies that $T$ is a finite conductor ring,
so we conclude that $T$  is a generalized GCD ring, again by Remark~\ref{disc150509a'}.
\end{proof}

\begin{thm}\label{thm150412a'}
Assume that the regular conductor square $\left( \square \right)$ satisfies condition (FP)
and every non-zero 2-generated ideal $I$ of $R$ satisfies condition (U1) from Definition~\ref{defn150522aa}.
If $C$ is contained in the Jacobson radical of $R$ (e.g., if $R$ is local), and $A$ is a  GCD domain, then $R$ and $T$ are GCD domains as well.
\end{thm}

\begin{proof}
By Proposition~\ref{prop150615a} and
Remark~\ref{disc150507az}, we assume that $\left( \square \right)$ is non-trivial.

Assume that $C$ is contained in the Jacobson radical of $R$, and $A$ is a  GCD domain.
Remark~\ref{disc150507az} shows that $A$ is a finite conductor ring.
Thus, Theorem~\ref{thm150412a} implies that $R$ is a finite conductor ring.
Also, by Proposition~\ref{prop150614a}, we know that $T$ is a domain, hence $R$ is also a domain by Remark~\ref{lem150509a}.

Let $I=(r,s)R$ such that $r,s\neq 0$.
As in the proof of Theorem~\ref{thm150412a}, we assume that $IT=T$, hence $I\supseteq IC=C$ by Lemma~\ref{lem150517a}\eqref{lem150517a1}.
In particular, we have $IA\cong I/CI\cong \Otimes AI$.
Moreover, we have $I\subsetneq C$; otherwise, the ideal $C=I$ would be 2-generated over $R$, contradicting Theorem~\ref{prop150419a}.
Thus, we have $IA\neq 0$.

According to Remark~\ref{disc150507az}, we need to show that the kernel $K$ of the natural map
$R^2\to(r,s)R$ is cyclic. 
Since $R$ is a finite conductor ring, we know that $K$ is finitely generated. 
Thus, by Nakayama's Lemma, to show that $K$ is cyclic, it suffices to show that $K/CK$ is cyclic over $A$.
Also, Lemma~\ref{lem150517a}\eqref{lem150517a2} implies that $\Tor 1AI=0$.

Consider the exact sequence
$$0\to K\to R^2\to I\to 0$$
and the induced long exact sequence in $\Tor {}A-$:
\begin{equation}\label{eq150520a}
0\to \underbrace{\Otimes AK}_{\cong K/CK}\to A^2\to \underbrace{\Otimes AI}_{\cong IA}\to 0.
\end{equation}
It follows that $K/CK$ is isomorphic to the kernel of the natural map $A^2\to IA$.
If $\eta_2(r)=0$ or $\eta_2(s)=0$, then it is straightforward to show that this kernel $K/CK$ is cyclic.
Otherwise, the fact that $A$ is a GCD domain implies that $K/CK$ is cyclic, again.
We conclude that $R$ is a GCD domain.
Thus, $T$ is a GCD domain as well, by Proposition~\ref{thm150506a'}.
\end{proof}

\begin{thm}\label{thm150412a''}
Assume that the regular conductor square $\left( \square \right)$ satisfies condition (FP)
and every non-zero 2-generated ideal $I$ of $R$ satisfies condition (U1) from Definition~\ref{defn150522aa}.
If $A$ is a generalized GCD ring, then so are $R$ and $T$.
\end{thm}

\begin{proof}
Assume that $A$ is a generalized GCD ring.
Note that $T$ is a B\'ezout domain, hence a generalized GCD domain, by Proposition~\ref{prop150614a} and Remarks~\ref{disc150507az} and~\ref{disc150509a'}.
In particular, Remark~\ref{lem150509a} implies that $R$ is a domain.

From Theorem~\ref{thm150412a}, we conclude that $R$ is a finite conductor ring.
Thus, it remains to show that $R$ is locally a GCD domain. 
Let $P\subset R$ be a prime ideal.
If $R_P\cong T_P$, e.g., if $C\not\subseteq P$, then the fact that $T$ is a generalized GCD ring implies that the localization $T_P\cong R_P$ is a GCD domain.

Assume for the rest of the proof that  $R_P\not\cong T_P$, thus, $C\subseteq P$. 
Consider the regular pullback square $(\square_P)$ from the proof of Lemma~\ref{lem150316a}.
It is straightforward to show that our assumptions on $(\square)$ imply that
$T_P$ is flat as an $R_P$-module, that the conductor ideal $C_P$ is a principal ideal of $T_P$,
and that every 2-generated ideal $I$ of $R_P$ is isomorphic to an ideal $I'\subseteq R_P$ such that $I'T_P=T_P$.
Furthermore, the ideal $C_P$ is contained in the Jacobson radical $P_P$ of $A_P$,  and
the ring $A_P$ is a GCD domain.
Thus, Theorem~\ref{thm150412a'} implies that $R_P$ is a GCD domain, as desired.
\end{proof}

\begin{rmk}\label{disc150507a}
Comparing the previous two results,
one might expect us to have a version of Theorem~\ref{thm150412a'} that does not assume that $C$ is contained in the Jacobson radical of $R$.
However, Example~\ref{ex150513a} below shows that this fails, even in a very nice case.
Note that this comes from~\cite[Example~6.12]{boynton:rpb},
and that the ring $R$ in this example is a generalized GCD domain, by Theorem~\ref{thm150412a''}, that is not a GCD domain.
\end{rmk}

\begin{ex}\label{ex150513a}
Set $T:=\bbq[X]$ and $B:=\bbq[i]$ with $\eta_1\colon T\to B$ the natural surjection.
Set $A:=\bbz[i]$ with $\iota_1\colon A\to B$ the inclusion map.
Note that the pullback determined by this data is of the form $\left(\boxtimes\right)$ from Remark~\ref{disc150513a}:
\begin{gather}
\begin{split}
\xymatrix{
R\ \ar@{^(->}[r]^-{\iota_2} \ar@{->>}[d]_-{\eta_2}
& \bbq[X]\ar@{->>}[d]^-{\eta_1} \\ 
\bbz[i] \ \ar@{^(->}[r]^-{\iota_1} &\bbq[i].}
\end{split}
\tag{$\boxtimes$}
\end{gather}
Moreover, it is straightforward to show that $R=\bbz+\bbz X+(X^2+1)\bbq[X]$;
in other words, the elements of $R$ are precisely the polynomials in $\bbq[X]$ such that,
when one divides by $X^2+1$, the division algorithm yields remainder $bX+c\in\bbz[X]$.
Also, we have $C=(X^2+1)\bbq[X]$.

We claim that $R$ is not a GCD domain.
By~\cite[Theorem 4.2(b)]{fontana:cglcgp}, it suffices to show that the map
$U(\bbq[X])\to \bbq[i]^*/U(\bbz[i])$ induced by $\eta_1$ is not surjective.
(One can also show directly that the ideal $I:=(X+1)R\bigcap(X^2+1)R$ is not principal. 
However, it is shorter to use~\cite{fontana:cglcgp}.)

Since $U(\bbq[X])=\bbq^*$ and $U(\bbz[i])=\{\pm 1\pm i\}$, the induced map in question is the natural one
$\bbq^*\to\bbq[i]^*/\{\pm 1\pm i\}$. Using the norm $N(a+bi)=a^2+b^2$, one checks readily that
the element $(i+i)\{\pm 1,\pm i\}$ is not in the image of this map; essentially, this boils down to the fact that $\sqrt 2\notin\bbq$.
\end{ex}

We continue with more transfer results.

\begin{thm}\label{thm150507ay}
Assume that the regular conductor square $\left( \square \right)$ satisfies condition (FP)
and every non-zero 2-generated ideal $I$ of $R$ satisfies condition (U2) from Definition~\ref{defn150522aa}.
Assume also that $A$ is  a domain.
If $R$ is a GCD domain, then $A$ is a GCD domain; the converse holds if $C$ is contained in the Jacobson radical of $R$.
Furthermore, these conditions imply that $T$ is also a GCD domain.
\end{thm}

\begin{proof}
Again, it suffices to assume that $R$ is a GCD domain and
show that $A$ is a GCD domain. Let $J=(a,b)A$ with $a,b\neq 0$, and consider the natural map $A^2\to J$.
We need to show that the kernel $K$ of this map is cyclic.

As in our previous proofs, let $f,g\in R$ be such that $\eta_2(f)=a$ and $\eta_2(g)=b$. 
Set $I=(f,g)R$, and consider the natural map $R^2\to I$.
Since $R$ is a GCD domain, the kernel $L$ of this map is cyclic over $R$.
Lemma~\ref{lem150517a}\eqref{lem150517a2} implies that $\Tor 1AI=0$.
Thus, applying $\Otimes A-$ to the short exact sequence
$$0\to L\to R^2\to I\to 0$$
yields another short exact sequence
$$0\to L/CL\to A^2\to I/CI\to 0.$$
It follows that  the kernel of the induced map $A^2\to I/CI$ is isomorphic to $L/CL$, which is cyclic. 

Case 1: $IT=T$. Then Lemma~\ref{lem150517a}\eqref{lem150517a1} implies that
$CI=C$, so the ideal 
$I/C$ is equal to $J$.
Thus,  the previous paragraph shows that $K\cong L/CL$, which is cyclic. 

Case 2:  $I$ is isomorphic to an ideal $I'$ such that $I'T=T$ via a multiplication map
$I'\xra[\cong]h I$  for a fixed $h\in R$.
In particular, $I'$ is 2-generated,
by elements $f',g'\in I'$ such that $f=f'h$ and $g=g'h$.
Set $a'=\eta_2(f')$ and $b'=\eta_2(g)$ and $x=\eta_2(h)$.
As in the proof of Theorem~\ref{thm150506c},
we have $a=xa'$
In particular, as $a\neq 0$, we have $x\neq 0$.

Since the ring $A$ is a domain, and  $x$ is non-zero,
the map $J':=(a',b')A\xra x(a,b)A=:J$ 
is injective;
it is surjective by construction. 
Since $I'T=T$, Case 1 implies that the kernel $K'$ of the natural map $A^2\to J'$ is cyclic.
Hence, the same is true for $K\cong K'$. This completes Case 2.

Now, we complete the proof.
By assumption, there are elements $r,s\in R$ and an ideal $U\subseteq R$ such that $UT=T=rT$ and
$I\xra[\cong] r rI=sU\xla[\cong] sU$. Case 2 implies that the kernel of the natural map $A^2\to rIA=sUA$ is cyclic.
As in the proof of Theorem~\ref{thm150506c}, the map $J=IA\xra r rIA$ is an isomorphism, so the kernel of the natural map $A^2\to J$ is cyclic as well,
as desired.
\end{proof}

\begin{thm}\label{thm150506cx}
Assume that the regular conductor square $\left( \square \right)$ satisfies condition (FP)
and every non-zero 2-generated ideal $I$ of $R$ satisfies condition (U2) from Definition~\ref{defn150522aa}.
Assume further that $A$ is  locally a domain and has finite Krull dimension.
Then $R$ is a generalized GCD ring if and only if $A$ is a generalized GCD ring; these conditions imply that $T$ is a generalized GCD ring.
\end{thm}

\begin{proof}
Again, we assume that $R$ is a generalized GCD ring and
show that $A$ is a generalized GCD ring.
Theorem~\ref{thm150506c} shows that $A$ is a finite conductor ring.
Thus, it remains to let $\fp$ be a prime ideal of $A$ and show that $A_{\p}$ is a GDC domain.
Note that $A_{\p}$ is a domain by assumption.

Set $P:=\nu_2^{-1}(\p)$ and consider the localized square $(\square_P)$ from Lemma~\ref{lem150316a}.
By assumption, $R_P$ is a GDC domain, so Theorem~\ref{thm150507ay} implies that 
the domain $A_\p\cong A_P$ is a GDC domain, as desired.
\end{proof}

\begin{thm}\label{thm150507ax}
Assume that the regular conductor square $\left( \square \right)$ satisfies condition (FP)
and every non-zero 2-generated ideal $I$ of $R$ satisfies condition (U2) from Definition~\ref{defn150522aa}.
Assume further that $A$ is a (possibly infinite) product $\prod_{\lambda\in\Lambda}A_\lambda$ of domains.
If $R$ is a generalized GCD ring, then each of the following rings is a generalized GCD ring: $T$, $A$, and $A_\lambda$.
Conversely, if $A$ is a generalized GCD ring, then each of the following rings is a generalized GCD ring: $T$, $R$, and $A_\lambda$.
\end{thm}

\begin{proof}
Note that $A$ is locally a domain. Indeed, since $A$ is a product of domains, it is straightforward to show that every principal ideal
of $A$ is a summand of $A$, hence projective, hence flat; now apply~\cite[Theorem 4.2.2]{glaz:ccr}.

Now, argue as in the proof of Theorem~\ref{thm150506cx}, using Theorem~\ref{thm150507a} instead of Theorem~\ref{thm150506c}.
\end{proof}

\begin{cor}\label{cor150517a'}
Consider the regular conductor square $(\boxtimes)$ from Remark~\ref{disc150513a}.
Then $R$ is a generalized GCD ring
if and only if each $A_i$ is so.
\end{cor}

\begin{cor}\label{cor150517b'}
Let $D$ be a domain, and let $E\subseteq D$ be a finite subset.
Then the ring of integer-valued polynomials
$\operatorname{Int}(E,D)$ is a generalized GCD ring
if and only if $D$ is so.
\end{cor}

\section{Quasi-coherent Rings}
\label{sec150612c}

We conclude with an investigation of the following property.

\begin{defn}\label{notn150509a''}
The ring $R$ is \emph{quasi-coherent} if every ideal of the form $(0:_Ra)$ is finitely generated, as is every intersection of finitely many principal ideals.
\end{defn}

\begin{disc}\label{disc150509a''}
B\'ezout domain $\implies$ coherent $\implies$ quasi-coherent $\implies$ finite conductor ring, and
GCD domain $\implies$ generalized GCD ring $\implies$ quasi-coherent.
Also, if $R$ is a finite product of domains, then every ideal of the form $(0:_Ra)$ is principal, hence finitely generated
\end{disc}

\begin{prop}\label{thm150506a''}
Consider the regular conductor square $\left( \square \right)$. Assume that $T$ is flat as an $R$-module. 
If $R$ is quasi-coherent, then $T$ is as well.
\end{prop}

\begin{proof}
Assume that $R$ is quasi-coherent. 
As in the proof of Proposition~\ref{thm150506a'}, one checks readily that the intersection of finitely many principal ideals of $T$ is principal.
Next, let $t\in T$, and consider the ideal $(0:_Tt)$. 
Let $c\in C$ be $T$-regular. It is straightforward to show that $(0:_Tct)=(0:_Tt)$,
so we assume without loss of generality that $t\in R$. 
Since $R$ is quasi-coherent, the ideal $(0:_Rt)$ is finitely generated over $R$.
By flatness, the ideal
$(0:_Tt)=(0:_Rt)T$ is finitely generated over $T$, as desired.
\end{proof}

\begin{disc}\label{disc150602a}
By~\cite[Proposition 2.4]{glaz:fcrzd}, a domain $R$ is quasi-coherent if and only if for each finitely generated ideal $I$
the module $\Hom IR\cong(R:_{Q(R)}I)$ is finitely generated. 
We  generalize this next for finite products of domains.
\end{disc}

\begin{lem}\label{lem150602b}
Let $D_1,\ldots,D_r$ be domains, and set $R=\prod_{i=1}^rD_r$. Then the following conditions are equivalent. 
\begin{enumerate}[\rm(i)]
\item \label{lem150602b1}
The ring $R$ is quasi-coherent.
\item \label{lem150602b2}
Each ring $D_i$ is quasi-coherent.
\item \label{lem150602b3}
For each finitely generated ideal $I$ of $R$, the module $\Hom IR$ is finitely generated over $R$.
\end{enumerate}
\end{lem}

\begin{proof}
Recall that every ideal $I$ of $R$ decomposes uniquely as $I=\prod_{i=1}^rI_i$ with each $I_i$ an ideal of $D_i$.
Moreover, the ideal $I$ is finitely generated over $R$ if and only each $I_i$ is finitely generated over $D_i$.
Also, given $D_i$-modules $M_i$ and $N_i$, with $M=\prod_{i=1}^rM_i$ and $N=\prod_{i=1}^rN_i$,
there is a natural isomorphism $\Hom MN\cong\prod_{i=1}^r\Hom[D_i]{M_i}{N_i}$ over $R$.

\eqref{lem150602b1}$\implies$\eqref{lem150602b2}
Assume that $R$ is quasi-coherent, and let $i\in\{1,\ldots,r\}$ be given.
Let $\{I_{i,j}\}_{j=1}^{n}$ be a finite set of principal ideals of $D_i$. 
For all $p\in\{1,\ldots,r\}\ssm\{i\}$ and all $j=1,\ldots,n$ set $I_{p,j}=R_p$, and set $I_j=\prod_{p=1}^rI_{p,j}$ which is a principal ideal of $R$. 
Since $R$ is quasi-coherent, the intersection $\bigcap_{j=1}^nI_j=\prod_{p=1}^r\left(\bigcap_{j=1}^nI_{p,j}\right)$ is finitely generated over $R$,
and it follows that the factor $\bigcap_{j=1}^nI_{i,j}$ is finitely generated over $D_i$, as desired.

\eqref{lem150602b2}$\implies$\eqref{lem150602b1}
Assume that each domain $D_i$ is quasi-coherent, and let $I_1,\ldots,I_n$ be principal ideals of $R$.
Since each ideal $I_j$ is of the form $I_j=\prod_{i=1}^rI_{i,j}$ where each $I_{i,j}$ is a principal ideal of $D_i$, 
the fact that $D_i$ is quasi-coherent implies that each intersection $\bigcap_{j=1}^nI_{i,j}$ is finitely generated over $D_i$,
so the intersection $\bigcap_{j=1}^nI_j=\prod_{p=1}^r\left(\bigcap_{j=1}^nI_{p,j}\right)$ is finitely generated over $R$.

\eqref{lem150602b2}$\implies$\eqref{lem150602b3}
Assume that each domain $D_i$ is quasi-coherent, and let $I=\prod_{i=1}^rI_i$ be a finitely generated ideal of $R$.
Then each $I_i$ is finitely generated over the quasi-coherent domain $D_i$, so Remark~\ref{disc150602a} implies that
$\Hom[D_i]{I_i}{D_i}$ is finitely generated over $D_i$.
It follows that the finite product
$\prod_{i=1}^r\Hom[D_i]{I_i}{D_i}\cong\Hom IR$ is finitely generated over $R$, as desired.

\eqref{lem150602b3}$\implies$\eqref{lem150602b2}
Assume that for each finitely generated ideal $I$ of $R$, the module $\Hom IR$ is finitely generated over $R$.
To show that $D_i$ is quasi-coherent, let $I_i$ be a finitely generated ideal of $D_i$.
For all $\p\in\{1,\ldots,n\}\ssm\{i\}$ set $I_p=D_p$, and set $I=\prod_{p=1}^nI_p$.
This ideal is finitely generated over $R$, so the module $\prod_{p=1}^r\Hom[D_p]{I_p}{D_p}\cong\Hom IR$ is finitely generated over $R$, by assumption.
It follows that the factor $\Hom[D_i]{I_i}{D_i}$ is finitely generated over $D_i$,
and Remark~\ref{disc150602a} implies that $D_i$ is quasi-coherent, as desired.
\end{proof}

\begin{disc}\label{disc150602b}
Notice that some implications in the previous result hold more generally than we have stated.
For instance, our proof readily shows that the implication~\eqref{lem150602b1}$\implies$\eqref{lem150602b2} holds for arbitrary products of rings that are not necessarily domains,
and the converse holds for finite products of rings that are not necessarily domains.
However, we are primarily interested in using the equivalence~\eqref{lem150602b1}$\iff$\eqref{lem150602b3} which seems to need $R$ to be a finite product of domains,
as we have assumed in the lemma.
\end{disc}

\begin{thm}\label{thm150601a}
Assume that the regular conductor square $\left( \square \right)$ satisfies condition (FP)
and every non-zero finitely generated ideal $I$ of $R$ satisfies condition (U1) from Definition~\ref{defn150522aa}.
Assume that $A$ is a finite product of domains.
If $A$ is  quasi-coherent, then $R$ and $T$ are quasi-coherent as well.
\end{thm}

\begin{proof}
As we have noted before, our assumptions imply that $R$ and $T$ are domains in this setting.
Assume that $A$ is  quasi-coherent.
By Proposition~\ref{thm150506a''}, it suffices to show that $R$ is quasi-coherent. 
Let $I$ be a finitely generated ideal of $R$. 
By Remark~\ref{disc150602a}, we need to show that $\Hom IR$ is finitely generated. 
Since every non-zero finitely generated ideal $I$ of $R$ satisfies condition (U1), we assume without loss of generality that $IT=T$. 

To show that $\Hom IR$ is finitely generated over $R$, it suffices by~\cite[Theorem~5.1.1(3)]{glaz:ccr}
to show 
that $\Hom IR/[C\cdot\Hom IR]$ is finitely generated over $A$
and 
that $\Otimes T{\Hom IR}$ is finitely generated over $T$. 
Since $A$ is quasi-coherent and a finite product of  domains, and the ideal $IA$ is finitely generated, the $A$-module
$$\Hom IR/[C\cdot\Hom IR]\cong\Hom[A]{IA}{A}$$ 
is finitely generated; see Lemma~\ref{lem150602a} for the isomorphism.

Next, consider the exact sequence
$$0\to \Hom IC\to\Hom IR\to\Hom IA\to 0$$
from the proof of Lemma~\ref{lem150602a}.
The isomorphism $\Hom IC\cong C\cong T$ from this lemma implies that this sequence has the following form:
$$0\to T\to\Hom IR\to\Hom IA\to 0.$$
As $T$ is flat over $R$, the preceding sequence yields the next exact sequence over $T$:
$$0\to \Otimes TT\to\Otimes T{\Hom IR}\to\Otimes T{\Hom IA}\to 0.$$
Since $R\to T$ is a flat epimorphism by Lemma~\ref{lem150510a}, this exact sequence has the following form:
$$0\to T\to\Otimes T{\Hom IR}\to\Otimes T{\Hom IA}\to 0.$$
We have already seen that $\Hom IA\cong\Hom[A]{IA}{A}$ is finitely generated over $A$, hence over $R$;
the isomorphism
is from Lemma~\ref{lem150602a}.
It follows that the module $\Otimes T{\Hom IA}$ is finitely generated over $T$.
Since $T$ is also finitely generated over $T$, the preceding exact sequence shows that $\Otimes T{\Hom IR}$
is finitely generated over $T$, as desired.
\end{proof}

\begin{thm}\label{thm150531a}
Assume that the regular conductor square $\left( \square \right)$ satisfies condition (FP)
and every non-zero finitely generated ideal $I$ of $R$ satisfies condition (U2) from Definition~\ref{defn150522aa}.
Assume that $A$ is a finite product of domains.
Then $R$ is  quasi-coherent if and only if $A$ is quasi-coherent; these conditions imply that $T$ is also quasi-coherent.
\end{thm}

\begin{proof}
Assume that  $R$ is  quasi-coherent.
Again, we only need to show that $A$ is  quasi-coherent.
Let $J=(\ol r_1,\ldots,\ol r_n)A$ be a non-zero finitely generated ideal of $A$;
here each $r_i$ is in $R$, and $\ol r_i$ is the residue of $r_i$ in $A$.
We need to show that $\Hom[A]JA$ is finitely generated over $A$, by Lemma~\ref{lem150602b}.
We argue as in the proofs of Theorems~\ref{thm150507a} and~\ref{thm150507ay}.
Consider the ideal $I=(r_1,\ldots,r_n)R$.

Case 1: $IT=T$.
In this case, Lemma~\ref{lem150602a} implies that 
$$\Hom IR/[C\cdot\Hom IR]\cong\Hom[A]{IA}{A}=\Hom[A]{J}{A}.$$ 
Since $R$ is a quasi-coherent  domain, the $R$-module $\Hom IR$ is finitely generated over $R$.
It follows that $\Hom[A]{J}A$ is finitely generated over $R$, hence over $A$, as desired. 
This concludes Case 1.

Case 2: The general case.
By condition (U2), there are elements $r,s\in R$ and an ideal $U\subseteq R$ such that $UT=T=rT$ and
$I\xra[\cong] r rI=sU\xla[\cong] sU$.
As in the proofs of Theorems~\ref{thm150506c} and~\ref{thm150507ay},
these induce isomorphisms $J\xra[\cong]{\ol r}\ol rJ=sUA\xla[\cong]{\ol s}UA$.
By Case 1, we know that $\Hom[A]{UA}{A}$ is finitely generated, hence so is $\Hom[A]JA\cong\Hom[A]{UA}{A}$.
\end{proof}

\begin{cor}\label{cor150517a''}
Consider the regular conductor square $(\boxtimes)$ from Remark~\ref{disc150513a}.
Then $R$ is quasi-coherent
if and only if each $A_i$ is so.
\end{cor}

\begin{cor}\label{cor150517b''}
Let $D$ be a domain, and let $E\subseteq D$ be a finite subset.
Then the ring of integer-valued polynomials
$\operatorname{Int}(E,D)$ is quasi-coherent
if and only if $D$ is~so.%
\end{cor}


\providecommand{\bysame}{\leavevmode\hbox to3em{\hrulefill}\thinspace}
\providecommand{\MR}{\relax\ifhmode\unskip\space\fi MR }
\providecommand{\MRhref}[2]{%
  \href{http://www.ams.org/mathscinet-getitem?mr=#1}{#2}
}
\providecommand{\href}[2]{#2}

\end{document}